\documentclass[12pt,twoside]{article}
\usepackage{amsmath,amssymb}
\usepackage{amsthm}
\usepackage{upref}
\usepackage{graphicx,amssymb}
\usepackage{color}
\numberwithin{equation}{section}

\theoremstyle{plain}
\newtheorem{theorem}{Theorem}[section]
\newtheorem{proposition}[theorem]{Proposition}

\theoremstyle{definition}
\newtheorem{remark}[theorem]{Remark}

\begin{document}

\newcommand{\eq}{equation}
\newcommand{\real}{\ensuremath{\mathbb R}}
\newcommand{\comp}{\ensuremath{\mathbb C}}
\newcommand{\rn}{\ensuremath{{\mathbb R}^n}}
\newcommand{\tn}{\ensuremath{{\mathbb T}^n}}
\newcommand{\rnp}{\ensuremath{\real^n_+}}
\newcommand{\rnn}{\ensuremath{\real^n_-}}
\newcommand{\Rn}{\ensuremath{{\mathbb R}^{n-1}}}
\newcommand{\Zn}{\ensuremath{{\mathbb Z}^{n-1}}}
\newcommand{\no}{\ensuremath{\nat_0}}
\newcommand{\ganz}{\ensuremath{\mathbb Z}}
\newcommand{\zn}{\ensuremath{{\mathbb Z}^n}}
\newcommand{\zom}{\ensuremath{{\mathbb Z}_{\Om}}}
\newcommand{\zOm}{\ensuremath{{\mathbb Z}^{\Om}}}
\newcommand{\As}{\ensuremath{A^s_{p,q}}}
\newcommand{\Bs}{\ensuremath{B^s_{p,q}}}
\newcommand{\Fs}{\ensuremath{F^s_{p,q}}}
\newcommand{\Fsr}{\ensuremath{F^{s,\rloc}_{p,q}}}
\newcommand{\nat}{\ensuremath{\mathbb N}}
\newcommand{\Om}{\ensuremath{\Omega}}
\newcommand{\di}{\ensuremath{{\mathrm d}}}
\newcommand{\sn}{\ensuremath{{\mathbb S}^{n-1}}}
\newcommand{\Ac}{\ensuremath{\mathcal A}}
\newcommand{\Acs}{\ensuremath{\Ac^s_{p,q}}}
\newcommand{\Bc}{\ensuremath{\mathcal B}}
\newcommand{\Cc}{\ensuremath{\mathcal C}}
\newcommand{\cc}{{\scriptsize $\Cc$}${}^s (\rn)$}
\newcommand{\ccd}{{\scriptsize $\Cc$}${}^s (\rn, \delta)$}
\newcommand{\Fc}{\ensuremath{\mathcal F}}
\newcommand{\Lc}{\ensuremath{\mathcal L}}
\newcommand{\Mc}{\ensuremath{\mathcal M}}
\newcommand{\Ec}{\ensuremath{\mathcal E}}
\newcommand{\Pc}{\ensuremath{\mathcal P}}
\newcommand{\Efr}{\ensuremath{\mathfrak E}}
\newcommand{\Mfr}{\ensuremath{\mathfrak M}}
\newcommand{\Mbf}{\ensuremath{\mathbf M}}
\newcommand{\Dbb}{\ensuremath{\mathbb D}}
\newcommand{\Lbb}{\ensuremath{\mathbb L}}
\newcommand{\Pbb}{\ensuremath{\mathbb P}}
\newcommand{\Qbb}{\ensuremath{\mathbb Q}}
\newcommand{\Rbb}{\ensuremath{\mathbb R}}
\newcommand{\vp}{\ensuremath{\varphi}}
\newcommand{\hra}{\ensuremath{\hookrightarrow}}
\newcommand{\supp}{\ensuremath{\mathrm{supp \,}}}
\newcommand{\ssupp}{\ensuremath{\mathrm{sing \ supp\,}}}
\newcommand{\dist}{\ensuremath{\mathrm{dist \,}}}
\newcommand{\unif}{\ensuremath{\mathrm{unif}}}
\newcommand{\ve}{\ensuremath{\varepsilon}}
\newcommand{\vk}{\ensuremath{\varkappa}}
\newcommand{\vr}{\ensuremath{\varrho}}
\newcommand{\pa}{\ensuremath{\partial}}
\newcommand{\oa}{\ensuremath{\overline{a}}}
\newcommand{\ob}{\ensuremath{\overline{b}}}
\newcommand{\of}{\ensuremath{\overline{f}}}
\newcommand{\LA}{\ensuremath{L^r\!\As}}
\newcommand{\LcA}{\ensuremath{\Lc^{r}\!A^s_{p,q}}}
\newcommand{\LcdA}{\ensuremath{\Lc^{r}\!A^{s+d}_{p,q}}}
\newcommand{\LcB}{\ensuremath{\Lc^{r}\!B^s_{p,q}}}
\newcommand{\LcF}{\ensuremath{\Lc^{r}\!F^s_{p,q}}}
\newcommand{\Lf}{\ensuremath{L^r\!f^s_{p,q}}}
\newcommand{\La}{\ensuremath{\Lambda}}
\newcommand{\Lob}{\ensuremath{L^r \ob{}^s_{p,q}}}
\newcommand{\Lof}{\ensuremath{L^r \of{}^s_{p,q}}}
\newcommand{\Loa}{\ensuremath{L^r\, \oa{}^s_{p,q}}}
\newcommand{\Lcoa}{\ensuremath{\Lc^{r}\oa{}^s_{p,q}}}
\newcommand{\Lcob}{\ensuremath{\Lc^{r}\ob{}^s_{p,q}}}
\newcommand{\Lcof}{\ensuremath{\Lc^{r}\of{}^s_{p,q}}}
\newcommand{\Lca}{\ensuremath{\Lc^{r}\!a^s_{p,q}}}
\newcommand{\Lcb}{\ensuremath{\Lc^{r}\!b^s_{p,q}}}
\newcommand{\Lcf}{\ensuremath{\Lc^{r}\!f^s_{p,q}}}
\newcommand{\id}{\ensuremath{\mathrm{id}}}
\newcommand{\tr}{\ensuremath{\mathrm{tr\,}}}
\newcommand{\trd}{\ensuremath{\mathrm{tr}_d}}
\newcommand{\trL}{\ensuremath{\mathrm{tr}_L}}
\newcommand{\ext}{\ensuremath{\mathrm{ext}}}
\newcommand{\re}{\ensuremath{\mathrm{re\,}}}
\newcommand{\Rea}{\ensuremath{\mathrm{Re\,}}}
\newcommand{\Ima}{\ensuremath{\mathrm{Im\,}}}
\newcommand{\loc}{\ensuremath{\mathrm{loc}}}
\newcommand{\rloc}{\ensuremath{\mathrm{rloc}}}
\newcommand{\osc}{\ensuremath{\mathrm{osc}}}
\newcommand{\pr}{\pageref}
\newcommand{\wh}{\ensuremath{\widehat}}
\newcommand{\wt}{\ensuremath{\widetilde}}
\newcommand{\ol}{\ensuremath{\overline}}
\newcommand{\os}{\ensuremath{\overset}}
\newcommand{\Li}{\ensuremath{\overset{\circ}{L}}}
\newcommand{\Ai}{\ensuremath{\os{\, \circ}{A}}}
\newcommand{\Ci}{\ensuremath{\os{\circ}{\Cc}}}
\newcommand{\dom}{\ensuremath{\mathrm{dom \,}}}
\newcommand{\SA}{\ensuremath{S^r_{p,q} A}}
\newcommand{\SB}{\ensuremath{S^r_{p,q} B}}
\newcommand{\SF}{\ensuremath{S^r_{p,q} F}}
\newcommand{\Hc}{\ensuremath{\mathcal H}}
\newcommand{\Nc}{\ensuremath{\mathcal N}}
\newcommand{\Lci}{\ensuremath{\overset{\circ}{\Lc}}}
\newcommand{\bmo}{\ensuremath{\mathrm{bmo}}}
\newcommand{\BMO}{\ensuremath{\mathrm{BMO}}}
\newcommand{\cm}{\\[0.1cm]}
\newcommand{\Aa}{\ensuremath{\os{\, \ast}{A}}}
\newcommand{\Ba}{\ensuremath{\os{\, \ast}{B}}}
\newcommand{\Fa}{\ensuremath{\os{\, \ast}{F}}}
\newcommand{\Aas}{\ensuremath{\Aa{}^s_{p,q}}}
\newcommand{\Bas}{\ensuremath{\Ba{}^s_{p,q}}}
\newcommand{\Fas}{\ensuremath{\Fa{}^s_{p,q}}}
\newcommand{\Ca}{\ensuremath{\os{\, \ast}{{\mathcal C}}}}
\newcommand{\Cas}{\ensuremath{\Ca{}^s}}
\newcommand{\Car}{\ensuremath{\Ca{}^r}}
\newcommand{\bl}{$\blacksquare$}

\begin{center}
{\Large Mapping properties of Fourier transforms, II}
\\[1cm]
{Hans Triebel}
\\[0.2cm]
Institut f\"{u}r Mathematik\\
Friedrich--Schiller--Universit\"{a}t Jena\\
07737 Jena, Germany
\\[0.1cm]
email: hans.triebel@uni-jena.de
\end{center}

\begin{abstract}
This is the direct continuation of the paper \cite{T21} using the same notation as there without further explanations. It deal with
continuous and compact mappings of the Fourier transform $F$ between some weighted function spaces on \rn.
\end{abstract}



\section{Basic properties}   \label{S1}
As already  mentioned in \cite[Problem 5.4]{T21} it is natural to deal with mapping properties of the Fourier transform $F$ in the context of the weighted spaces
\begin{\eq}  \label{1.1}
\Bs (\rn, w_\alpha), \qquad s\in \real \quad \text{and} \quad 0<p,q \le \infty,
\end{\eq}
as introduced in \cite[Definition 4.3]{T21}  where again
\begin{\eq}   \label{1.2}
w_\alpha (x) = (1 +|x|^2 )^{\alpha/2}, \qquad x\in \rn, \quad \alpha \in \real.
\end{\eq}
In addition to the isomorphic mapping  $f \mapsto w_\alpha f$,
\begin{\eq}   \label{1.3}
\| w_\alpha f \, | \Bs (\rn) \| \sim \|f \, | \Bs (\rn, w_\alpha) \|
\end{\eq}
according to \cite[(4.6)]{T21} the lifting
\begin{\eq}   \label{1.4}
\| (w_\alpha \wh{f} )^\vee | \Bs (\rn, w_\beta) \| \sim \| f \, | B^{s+\alpha}_{p,q} (\rn, w_\beta) \|, \qquad \alpha \in \real, \quad
\beta \in \real,
\end{\eq}
will be of some use for us, \cite[Theorem 6.5, pp.\,265--266]{T06} and the references given there. We concentrate again on
\begin{\eq}   \label{1.5}
B^s_p (\rn, w_\alpha) = B^s_{p,p} (\rn, w_\alpha), \qquad s\in \real, \quad \alpha \in \real, \quad 1<p<\infty,
\end{\eq}
and its special case
\begin{\eq}   \label{1.6}
H^s (\rn,  w_\alpha) = B^s_2 (\rn, w_\alpha) = B^s_{2,2} (\rn, w_\alpha), \qquad s\in \real, \quad \alpha \in \real.
\end{\eq}

\begin{proposition}   \label{P1.1}
Let $s\in \real$ and $\alpha \in \real$. Then the Fourier transform
\begin{\eq}   \label{1.7}
F: \quad H^s (\rn, w_\alpha) \hra H^\alpha (\rn, w_s)
\end{\eq}
is an isomorphic mapping,
\begin{\eq}   \label{1.8}
F H^s (\rn, w_\alpha) = H^\alpha (\rn, w_s).
\end{\eq}
\end{proposition}

\begin{proof}
Let $f \in H^s (\rn, w_\alpha)$. Then it follows from \eqref{1.4} (with $L_2 (\rn, w_s)$ in place of $\Bs (\rn, w_\beta)$),
 $FL_2 (\rn) = L_2 (\rn)$, and \eqref{1.3} that
\begin{\eq}   \label{1.9}
\begin{aligned}
\| \wh{f} \, | H^\alpha (\rn, w_s) \| & \sim \|(w_\alpha f )^\wedge | L_2 (\rn, w_s) \| \\
&\sim \| w_s (w_\alpha f )^\wedge \, | L_2 (\rn) \| \\
 &\sim \| w_\alpha f \, | H^s (\rn) \| \\ &\sim
\| f \, | H^s (\rn, w_\alpha) \|.
\end{aligned}
\end{\eq}
Conversely, for any $g\in H^\alpha (\rn, w_s)$ there is an $f\in H^s (\rn, w_\alpha)$ with $\wh{f} =g$ and a  counterpart of 
\eqref{1.9}. This proves the proposition.
\end{proof}

\begin{remark}   \label{R1.2}
In particular,
\begin{\eq}   \label{1.10}
F H^s (\rn, w_s)  = H^s (\rn, w_s), \qquad s \in \real,
\end{\eq}
may be considered as the weighted extension of
\begin{\eq}   \label{1.11}
F L_2 (\rn) =L_2 (\rn), \qquad L_2 (\rn) = H^0 (\rn, w_0).
\end{\eq}
\end{remark}

Now it is quite clear that the role played by $L_2 (\rn)$ and $L_2 (\rn, w_\alpha)$, $\alpha >0$, in the theory of compact mappings of
$F$ between unweighted spaces as developed in \cite{T21}
is now taken over by $H^s (\rn, w_s)$, $s\in \real$, and $H^s (\rn, w_{s+\alpha})$ where again the degree of compactness is measured  in 
terms of entropy numbers as recalled in \cite[Definition 4.1]{T21} including related referenced to the literature. One may ask for
entropy numbers of compact mappings
\begin{\eq}  \label{1.12}
F: \quad B^{s_1}_p (\rn, w_{\alpha_1}) \hra B^{s_2}_p (\rn, w_{\alpha_2}),
\end{\eq}
for fixed weights, which means $\alpha_1 = \alpha_2$, for fixed smoothness,
which means $s_1 = s_2$, or for a mixture of both. But this will not be done here in detail. We add  now  a comment to the case of
fixed weights and shift  the more interesting task for fixed smoothness to the next section.

\begin{theorem}   \label{T1.3}
Let $-\infty < s_2 <s< s_1 <\infty$. Then
\begin{\eq}   \label{1.13}
F: \quad H^{s_1} (\rn, w_s ) \hra H^{s_2} (\rn, w_s)
\end{\eq}
is compact and
\begin{\eq}   \label{1.14}
e_k (F) \sim
\begin{cases}
k^{-\frac{\sigma_2}{n}} &\text{if $\sigma_2 < \sigma_1$}, \\
\big( \frac{k}{\log k} \big)^{-\frac{\sigma_2}{n}} &\text{if $\sigma_2 = \sigma_1$}, \\
k^{- \frac{\sigma_1}{n}} &\text{if $\sigma_2 > \sigma_1$},
\end{cases}
\end{\eq}
$2\le k \in \nat$, where $s_1 =s+ \sigma_1$ and $s_2 = s - \sigma_2$.
\end{theorem}

\begin{proof} 
By \eqref{1.8} one has
\begin{\eq}    \label{1.15}
F H^{s_1} (\rn, w_s) = H^s (\rn, w_{s_1} ).
\end{\eq}
This extends \cite[(4.37)]{T21} from $s=0$ to $s\in \real$. Then one can argue as there, relying now on \cite[Proposition 4.5]{T21}.
\end{proof}

\begin{remark}    \label{R1.4}
This extends \cite[Theorem 4.8(ii)]{T21} from $s=0$ to $s\in \real$. It is quite clear that there are related counterparts of 
\cite[Theorem 4.8(iii),(iv)]{T21} for the compact mappings
\begin{\eq}   \label{1.16}
F: \quad B^{s+ \sigma_1}_p (\rn, w_s) \hra B^{s-\sigma_2}_p (\rn, w_s), \qquad 1<p<\infty, \quad s\in \real,
\end{\eq}
with
\begin{\eq}   \label{1.17}
\begin{cases}
\sigma_1 >d^n_p, \ \sigma_2 >0 &\text{if $1<p \le 2$}, \\
\sigma_1 >0, \ \sigma_2 > |d^n_p| &\text{if $2\le p <\infty$},
\end{cases}
\end{\eq}
where as there
\begin{\eq}   \label{1.18}
d^n_p = 2n \big( \frac{1}{p} - \frac{1}{2} \big), \qquad n\in \nat, \quad 1<p<\infty.
\end{\eq}
\end{remark}

\section{Main assertions}    \label{S2}
So far we dealt in Theorem \ref{T1.3} and in the indicated generalizations in \eqref{1.16} with the same weight both in the source
spaces and in the target spaces. The outcome is apparently a rather straightforward generalization of corresponding assertions in
\cite{T21} for the unweighted spaces. The question arises what happens if the weights in the source spaces and in the target spaces
are different. For this purpose it seems to be reasonable to fix first not only the integrability parameter $p$ with $1<p<\infty$, but
also the smoothness $s\in \real$ and to ask, suggested by \eqref{1.10}, for compact mappings
\begin{\eq}   \label{2.1}
F: \quad B^s_p (\rn, w_{s+\alpha}) \hra B^s_p (\rn, w_{s-\beta}),
\end{\eq}
$1<p<\infty$, $s\in \real$, $\alpha >0$, $\beta >0$. First we deal with the case $p=2$ using the notation recalled in \eqref{1.6}.

\begin{proposition}   \label{P2.1}
Let $s\in \real$, $\alpha >0$ and $\beta >0$. Then
\begin{\eq}   \label{2.2}
F: \quad H^s (\rn, w_{s+\alpha} ) \hra H^s (\rn, w_{s-\beta} )
\end{\eq}
is compact and
\begin{\eq}   \label{2.3}
e_k (F) \sim
\begin{cases}
k^{-\frac{\alpha}{n}} &\text{if $\alpha < \beta$}, \\
\big( \frac{k}{\log k} \big)^{-\frac{\alpha}{n}} &\text{if $\alpha = \beta$}, \\
k^{- \frac{\beta}{n}} &\text{if $\alpha > \beta$},
\end{cases}
\end{\eq}
$2\le k \in \nat$.
\end{proposition}

\begin{proof}
By \cite[Corollary 4.7(ii)]{T21} one has that
\begin{\eq}   \label{2.4}
\id: \quad L_2 (\rn, w_\alpha) \hra H^{-\sigma} (\rn)
\end{\eq}
with $\alpha >0$ and $\sigma >0$ is compact and
\begin{\eq}   \label{2.5}
e_k (F) \sim
\begin{cases}
k^{-\frac{\sigma}{n}} &\text{if $\sigma < \alpha$}, \\
\big( \frac{k}{\log k} \big)^{-\frac{\sigma}{n}} &\text{if $\sigma = \alpha$}, \\
k^{- \frac{\alpha}{n}} &\text{if $\sigma >\alpha$},
\end{cases}
\end{\eq}
$2\le k \in \nat$. Then it follows from the two isomorphisms \eqref{1.3} and \eqref{1.4} that
\begin{\eq}   \label{2.6}
\id: \quad H^{s+\alpha} (\rn, w_s) \hra H^s (\rn, w_{s-\beta} ), \qquad s\in \real,
\end{\eq}
with $\alpha >0$ and $\beta >0$ is compact and
\begin{\eq}   \label{2.7}
e_k (F) \sim
\begin{cases}
k^{-\frac{\alpha}{n}} &\text{if $\alpha <\beta$}, \\
\big( \frac{k}{\log k} \big)^{-\frac{\alpha}{n}} &\text{if $\alpha = \beta$}, \\
k^{- \frac{\beta}{n}} &\text{if $\alpha > \beta$},
\end{cases}
\end{\eq}
$2\le k \in \nat$. Now one obtains \eqref{2.3} from \eqref{2.6}, \eqref{2.7} and
\begin{\eq}   \label{2.8}
F H^s (\rn, w_{s+\alpha} ) = H^{s+\alpha} (\rn, w_s)
\end{\eq}
according to Proposition \ref{P1.1}.
\end{proof}

We extend Proposition \ref{P2.1} and ask for conditions ensuring that $F$ in \eqref{2.1} is compact. Let $d^n_p = 2n (\frac{1}{p} - \frac{1}{2})$
be as in \eqref{1.18}. 

\begin{theorem}   \label{T2.2}
{\em (i)} Let $1<p \le 2$, $s\in \real$ and $\alpha >0$, $\beta >0$. Then
\begin{\eq}   \label{2.9}
F: \quad B^s_p (\rn, w_{s+\alpha} ) \hra B^s_p (\rn, w_{s- d^n_p -\beta})
\end{\eq}
is compact and
\begin{\eq}   \label{2.10}
e_k (F) \le c \
\begin{cases}
k^{-\frac{\alpha}{n}} &\text{if $\alpha < \beta$}, \\
\big( \frac{k}{\log k} \big)^{-\frac{\alpha}{n}} (\log k)^{\frac{1}{p} - \frac{1}{2}} &\text{if $\alpha = \beta$}, \\
k^{- \frac{\beta}{n}} &\text{if $\alpha > \beta$},
\end{cases}
\end{\eq}
for some $c>0$ and all $2\le k \in \nat$. 
\cm
{\em (ii)} Let $2 \le p <\infty$, $s\in \real$ and $\alpha >0$, $\beta >0$. Then
\begin{\eq}   \label{2.11}
F: \quad B^s_p (\rn, w_{s+ |d^n_p| + \alpha} ) \hra B^s_p (\rn, w_{s-\beta})
\end{\eq}
is compact and
\begin{\eq}   \label{2.12}
e_k (F) \le c \
\begin{cases}
k^{-\frac{\alpha}{n}} &\text{if $\alpha < \beta$}, \\
\big( \frac{k}{\log k} \big)^{-\frac{\alpha}{n}} (\log k)^{\frac{1}{2} - \frac{1}{p}} &\text{if $\alpha = \beta$}, \\
k^{- \frac{\beta}{n}} &\text{if $\alpha > \beta$},
\end{cases}
\end{\eq}
for some $c>0$ and all $2\le k \in \nat$. 
\end{theorem}

\begin{proof}
{\em Step 1.} The case $p=2$ is covered by Proposition \ref{P2.1} (even with equivalence instead of an estimate from above).
\cm
{\em Step 2.} Let $1<p<2$. In modification of \eqref{2.9} we ask first under which conditions $F$ in \eqref{2.1},
\begin{\eq}   \label{2.13}
F: \quad B^s_p (\rn, w_{s+\alpha} ) \hra B^s_p (\rn, w_{s-\beta})
\end{\eq}
is compact. By the isomorphism \eqref{1.3} and the well--known embedding for unweighted spaces we have the continuous embedding
\begin{\eq}   \label{2.14}
\id_1: \quad B^s_p (\rn, w_{s+\alpha}) \hra H^{s-n(\frac{1}{p} - \frac{1}{2})} (\rn, w_{s+\alpha}).
\end{\eq}
This shows, combined with with
\begin{\eq}   \label{2.15}
F H^{s- n(\frac{1}{p} - \frac{1}{2})} (\rn, w_{s+\alpha})  = H^{s+\alpha} (\rn, w_{s-n(\frac{1}{p} - \frac{1}{2})}),
\end{\eq}
covered by Proposition \ref{P1.1} that \eqref{2.13} can be reduced to the question under which conditions
\begin{\eq}   \label{2.16}
\id_2: \quad H^{s+ \alpha} (\rn, w_{s-n (\frac{1}{p} - \frac{1}{2})}) \hra B^s_p (\rn, w_{s-\beta})
\end{\eq}
is compact. For this purpose we specify \cite[Proposition 4.5, (4.7)--(4.14)]{T21} to $\id_2$. This requires in the notation used there $\alpha >0$,
\begin{\eq}   \label{2.17}
\delta= \alpha + n \big( \frac{1}{p} - \frac{1}{2} \big) = \alpha + \frac{1}{2} d^n_p, \qquad \vr = \frac{\delta}{n} >0,
\end{\eq}
and
\begin{\eq}   \label{2.18}
\frac{1}{p} < \frac{1}{p_*} = \frac{1}{2} + \frac{1}{n} \big( s - \frac{n}{p} + \frac{n}{2} -s + \beta \big) = 1 - \frac{1}{p} + \frac{\beta}{n}
\end{\eq}
resulting in $\beta > 2n (\frac{1}{p} - \frac{1}{2} ) = d^n_p$. Then it follows from \cite[Proposition 4.5]{T21} that $\id_2$ in \eqref{2.16} is 
compact. Replacing there $s_1$ by $s+\alpha$, $s_2$ by $s$, $\alpha$ by $\beta - n (\frac{1}{p} - \frac{1}{2}) = \beta - \frac{1}{2} d^n_p$, $p_1$
by $2$ and  $p_2$ by $p$ one obtains for the corresponding entropy numbers
\begin{\eq}   \label{2.19}
e_k (\id_2) \le c \
\begin{cases}
k^{-\frac{\alpha}{n}} &\text{if $\alpha + \frac{1}{2} d^n_p < \beta - \frac{1}{2} d^n_p$}, \\
\big( \frac{k}{\log k} \big)^{-\frac{\alpha}{n}} (\log k)^{\frac{1}{p} - \frac{1}{2}} &\text{if $\alpha  + \frac{1}{2} d^n_p= \beta - 
\frac{1}{2} d^n_p$}, \\
k^{- \frac{\beta}{n} + \frac{1}{p} - \frac{1}{2} + \frac{1}{p} - \frac{1}{2}} &\text{if $\alpha + \frac{1}{2} d^n_p > \beta - \frac{1}{2} d^n_p$}.
\end{cases}
\end{\eq}
Then \eqref{2.10} follows from \eqref{2.14}--\eqref{2.16} and \eqref{2.19} replacing there $\beta$ by $\beta + d^n_p$.
\cm
{\em Step 3.} Let $2<p<\infty$. As in Step 4 of the proof of \cite[Theorem 4.8]{T21} we rely on the duality
\begin{\eq}   \label{2.20}
B^s_p (\rn, w_\sigma)' = B^{-s}_{p'} (\rn, w_{-\sigma}), \qquad 1<p<\infty, \quad \frac{1}{p} + \frac{1}{p'} =1,
\end{\eq}
$s\in \real$, $\sigma \in \real$, in the framework of the dual pairing  $ \big( S(\rn), S' (\rn) \big)$. The isomorphism \eqref{1.3}
shows that also the weighted spaces $B^s_p (\rn, w_\sigma)$ are 
isomorphic to $\ell_p$. Then one can apply the duality theory for entropy numbers as described there. 
Using $d^n_{p'} = - d^n_p$ one obtains part (ii) of the above theorem by the indicated duality from part (i).
\end{proof}

\begin{remark}   \label{R2.3}
The outcome may justify to deal not only with unweighted spaces as in the Theorems \cite[Theorem 4.8]{T21} but also with the above weighted spaces for fixed $s\in \real$ and $p$ with $1<p<\infty$. The typical gap $d^n_p$ for the smoothness in the
unweighted case is now shifted to the weights.
\end{remark}

\begin{remark}   \label{R2.4}
If $1<p \le 2$, $\alpha \not= \beta$ and $0<q_\alpha, q_\beta \le \infty$ then the estimate \eqref{2.10} remains valid for
\begin{\eq}   \label{2.21}
F: \quad B^s_{p, q_\alpha} (\rn, w_{s+\alpha}) \hra B^s_{p, q_\beta} (\rn, w_{s- d^n_p -\beta} ).
\end{\eq}
Similarly for \eqref{2.11}, \eqref{2.12}. This follows by real interpolation in the same way as in \cite[Corollary 4.10]{T21}. 
For this purpose one should first shift the $s$--dependence for the weights to the exponents of the estimates for the corresponding
entropy numbers and use afterwards that the interpolation of $\Bs (\rn, w_\gamma)$ for a fixed weight $w_\gamma$ is the same as for 
their unweighted ancestors $\Bs (\rn)$.
\end{remark}

\section{Spectral theory}   \label{S3}
Let $K: \, B \hra B$ be a linear compact operator in a complex infinitely dimensional quasi--Banach space $B$. Let $e_k (B)$ be its
entropy numbers. Let $\{ \lambda_k (K)\}$ be the sequence of all non--zero eigenvalues of $K$, repeated according to their algebraic
multiplicity and naturally ordered by magnitude. We used in \cite{T22} Carl's observation
\begin{\eq}  \label{3.1}
|\lambda_k (K)| \le \sqrt{2} e_k (K), \qquad k\in \nat,
\end{\eq}
to study the distribution of eigenvalues for distinguished so--called Fourier operators based on corresponding assertions in 
\cite{T21}. Details, explanations and references may be found in \cite{T22}. This will not be repeated here. One can now use Theorem
\ref{T2.2} and other assertions obtained in this note to extend this theory to further classes of operators. This will not be done. But we
illustrate what can be expected by a simple example. Let as before $w_\sigma (x) = (1+|x|^2 )^{\sigma /2}$, $\sigma \in \real$, $x\in
\rn$. Recall that $H^0 (\rn, w_\sigma) = L_2 (\rn, w_\sigma)$. 
                                                                 
\begin{proposition}   \label{P3.1}
Let $\alpha >0$. Then $K_\alpha = w_{-\alpha} \circ F \circ w_{-\alpha}$ is a compact operator in $L_2 (\rn)$ and for some $c>0$,
\begin{\eq}   \label{3.2}
|\lambda_k (K_\alpha )| \le c \Big( \frac{k}{\log k} \Big)^{- \frac{\alpha}{n}}, \qquad 2 \le k \in \nat.
\end{\eq}
\end{proposition}

\begin{proof}
One can decompose $K_\alpha$ into the two isomorphic mappings $f \mapsto w_{-\alpha} f$ from $L_2 (\rn)$ onto $L_2 (\rn, w_\alpha)$
and from $L_2 (\rn, w_{-\alpha})$ onto $L_2 (\rn)$ combined with
\begin{\eq}   \label{3.3}
F: \quad L_2 (\rn, w_\alpha) \hra L_2 (\rn, w_{-\alpha} ).
\end{\eq}
Then \eqref{3.2} follows from \eqref{3.1} and Proposition \ref{P2.1} with $s=0$ and $\alpha = \beta >0$.
\end{proof}

\begin{remark}   \label{R3.2}
As said, Proposition \ref{P3.1} may serve as an example which type of assertions can be expected if one deals more systematically
with problems of this type based on \cite{T21}, \cite{T22}.
\end{remark}

\end{document}